\documentclass[12pt]{article}
\usepackage{amsmath,amsfonts,amssymb,amsthm,amscd}
\title{Stability and Grothendieck}
\date{\today}
\author{Anand Pillay}
\newtheorem{Theorem}{Theorem}[section]
\newtheorem{Proposition}[Theorem]{Proposition}
 
\newtheorem{Remark}[Theorem]{Remark}

\newcommand{\R}{\mathbb R}

\begin{document}
\maketitle

\begin{abstract} We  interpret Grothendieck's double limit characterization of weak relative compactness \cite{Grothendieck}  in the model theoretic setting as: $\phi(x,y)$ does not have the order property in $M$ iff and only if every complete $\phi(x,y)$-type $p(x)$ over $M$ is {\em generically stable}. We give a proof and point out the connection with \cite{Pillay}.

\end{abstract}

\section{Introduction}
This note is a commentary on the model-theoretic interpretation of Grothendieck's double limit characterization of weak relative compactness, after having read  Ita\"{i} Ben-Yaacov's short paper \cite{Ben-Yaacov}  on the topic,  in the model theory seminar at Notre Dame.  Thanks to Gabriel Conant, Sergei Starchenko and members of the Notre Dame model theory seminar for discussions. 

\vspace{5mm}
\noindent
The Grothendieck result, Theorem 6 in \cite{Grothendieck},  is that if $X$ is a compact (Hausdorff) space, $X_{0}$ is a dense subset of $X$, and $A$ is a subset of $C(X)$, the (necessarily) bounded continuous functions on $X$, then the following are equivalent:
\newline
(a) the closure of $A$ in $C(X)$ with respect to the weak topology on $C(X)$ is compact (with respect to this weak topology on $C(X)$), and
\newline
(b) $A$ is bounded in $C(X)$, and if $f_{i}\in A$ and $x_{i}\in X_{0}$ (i=1,2,...), then  if both $lim_{i}lim_{j}f_{i}(x_{j})$ and $lim_{j}lim_{i}f_{i}(x_{j})$ exist then they are equal.

\vspace{5mm}
The classical model theory context is where $M$ is an $L$-structure, $\phi(x,y)$ an $L$-formula, $\phi^{*}(y,x)$ the same formula but with $y$ as the ``variable 
variable", $X = S_{\phi^{*}}(M)$, $X_{0} = \{tp_{\phi^{*}}(b/M): b\in M\}$ (the realized types), and $A$ the set of (continuous) functions from $X$ to $2$ given by formulas $\phi(a,y)$ for $a\in M$. We let $M^{*}$ be a saturated elementary extension of $M$.   Condition (b) says that $\phi(x,y)$ does not have the 
order property in $M$, namely there do NOT exist $a_{i},b_{i}$ in $M$ for $i<\omega$ such that for all $i,j$, $M\models \phi(a_{i},b_{j})$ iff $i\leq j$ or for all $i,j$, $M\models \neg\phi(a_{i}, b_{j})$ iff $i\leq j$.   
Now for condition (a):  Weak compactness  of a subset $B$ of the set of continuous functions from $X$ to $2$ is equivalent to pointwise compactness of $B$. Hence condition (a) says that whenever $f\in 2^{X}$ is in the closure of $A\subseteq 2^{X}$  (in the pointwise convergence, equivalently Tychonoff, topology on the space $2^{X}$ of {\em all} functions from $X$ to $2$) then $f$ is continuous, i.e. given by a $\phi^{*}$-formula  (namely a finite Boolean combination of $\phi(a,y)$'s  for $a\in M$).  

We will give a  quick  proof of this equivalence of (a) and (b) in the model-theoretic context (see Proposition 2.2 below).  In fact the proof will be Grothendieck's one (proof of (d) implies (a) of Theorem 2 in \cite{Grothendieck}), which he says is based on an idea of Eberlein, but amusingly, is also essentially  the proof of Proposition 3.1 from \cite{Pillay} where we proved that if $\phi(x,y)$ does not have the order property in $M$, and $a,b\in M^{*}$ then $tp_{\phi}(a/M,b)$ is finitely satisfiable in $M$ iff $tp_{\phi^{*}}(b/M,a)$ is finitely satisfiable in $M$.

As we point out, conditions (a), (b) in the model-theoretic context imply (and are equivalent to) the statement that every $p(x)\in S_{\phi}(M)$ has an extension $p'\in S_{\phi}(M^{*})$ which is both finitely satisfiable in and definable over $M$ (where moreover the $\phi(x,y)$-definition for $p$ is a $\phi^{*}$-formula over $M$). 

In \cite{NIPII}  we defined a {\em complete type}  $p(x)$ over a model $M$ to be generically stable if $p$ has an extension to a complete type $p'$ over $M^{*}$ which is finitely satisfiable in, and definable over $M$.  Under the assumption that $T$ has $NIP$, we showed in \cite{NIPII},  that generically stable complete types $p(x)$ have  additional properties, such as $p'$ being the unique nonforking extension of $p$. Subsequently in \cite{Pillay-Tanovic} an appropriate stronger definition of generic stability (of a complete type) was given in an arbitrary theory, in such a way that the additional properties are satisfied. 

So morally, the model-theoretic meaning of the Grothendieck theorem is that the formula $\phi(x,y)$ does not have the order property in $M$ if and only if every complete $\phi$-type $p(x)\in S_{\phi}(M)$ is generically stable.  And this was already implicit in \cite{Pillay} where we obtained  generic stability of every complete type over $M$ from ``$M$ has no order" (i.e. no formula $\phi(x,y)$ has the order property in $M$).  We will investigate later to what extent we can deduce the  stronger notions of generic stability from not the order property in $M$.
\\

Let is briefly give definitions of some of the functional analysis notions. Given a compact space $X$ and the real Banach space $C(X)$ of  continuous functions, let $L(C(X),\R)$ be the space of bounded linear functions on $C_{b}(X)$.  The weak topology on $C(X)$ is the one whose basic open neiighbourhoods of a point $f_{0}$ are of the form $\{f\in C(X): |g_{1}(f-f_{0})| < \epsilon, ..., |g_{r}(f-f_{0})|<\epsilon\}$ for some $\epsilon > 0$ and some finite set $g_{1},..,g_{n}$ from $L(C(X),\R)$. 
A basic fact (see Lemma D.3 of \cite{Kerr-Li}) is that a bounded subset $B$ of $C(X)$ is compact in the weak topology iff $B$ is compact in the ``pointwise convergence" topology, namely in the product topology on $D^{X}$ for a suitable compact interval in $\R$. 
It follows for example that if $B$ is a subset of the continuous functions from $X$ to $\{0,1\}$, then the closure of $B$ in $C(X)$ (which will be contained of course in $2^{X}$) in the  weak topology on $C(X)$ is compact iff the closure of $B$ in the space $2^{X}$ with the product topology consists of continuous functions.

\section{Theorem,  proof, and discussion}
Let us first fix notation. $\phi(x,y)$ is an $L$-formula, and $M$ an $L$-structure. $S_{\phi}(M)$ denotes the space of complete $\phi$-types over $M$ (in variable $x$).  $\phi^{*}(y,x)$ is $\phi(x,y)$ and $S_{\phi^{*}}(M)$ denotes the space of complete $\phi^{*}$-types over $M$ (in  variable $y$).  We let $X$ denote the space $S_{\phi^{*}}(M)$. A $\phi$-formula over $M$ is a (finite) Boolean combination of formulas $\phi(x,b)$ for $b$ in $M$. The $\phi$-formulas pick out the clopen subsets of $S_{\phi}(M)$. Likewise for $\phi^{*}$-formulas and $S_{\phi^{*}}(M)$.  Let $M^{*}$ be a saturated elementary extension of $M$ and $M^{**}$ a saturated elementary extension of $M^{*}$.
Any formula $\phi(a,y)$ with $a\in M$ can be evaluated at any $q\in X$ (i.e. truth value of $\phi(a,b)$ for some/any realization $b$ of $q$), and by definition the corresponding map $X\to 2$ is continuous. 

\begin{Remark} Let $f:X\to 2$ be in the closure of the set of functions $X\to 2$ given by formulas $\phi(a,y)$ for $a\in M$ (in the product topology on $2^{X}$). Then there is $a^{*}\in M^{**}$ such that $tp(a^{*}/M^{*})$ is finitely satisfiable in $M$ (in  particular $M$-invariant), and for $q\in X$, $f(q)$ is the value (true or false) of $\phi(a^{*},b)$ for some/any $b\in M^{*}$ realizing $q$.  Conversely any such $a^{*}$ yields in this way a function $X\to 2$ in the closure of the set of functions given by $\phi(a,y)$ for $a\in M$. 
\end{Remark}

Modulo the discussion of weak compactness in Section 1, the equivalence of (a) and (b) below is precisely the statement of Grothendieck's theorem in the classical model-theoretic environment.

\begin{Proposition} The following are equivalent.
\newline
(a) If $f\in 2^{X}$ is in the closure in the pointwise convergence topology (equivalently product topology) of the set of functions given by $\phi(a,y)$ for $a\in M$, then $f$ is continuous, so given by a $\phi^{*}$-formula over $M$.
\newline
(b) $\phi(x,y)$ does not have the order property in $M$.
\end{Proposition} 
\begin{proof} First the ``easy" direction (a) implies (b). Assume (a), and suppose (b) fails, namely  $\phi$ does not have the order property in $M$ witnessed without loss of generality by  $a_{i}, b_{i}$ in $M$ for $i<\omega$ such that $M\models \phi(a_{i},b_{j})$ iff $i\leq j$.  By (a) there is a subsequence $a_{j_{i}}$ $i<\omega$ such that the functions $\phi(a_{j_{i}},y)$ converge pointwise to some $\phi^{*}$-formula $\psi(y)$ over $M$. This means that for every $b\in M^{*}$, the value of $\psi(b)$  is the eventual value of $\phi(a_{j_{i}},b)$.  Clearly we have to have $\models \neg\psi(b_{i})$ for all $i$, so  by compactness we can find $b\in  M^{*}$ such that $\models\neg\psi(b)$ and $\models \phi(a_{i},b)$ for all $i$. This is a contradiction. 
\newline
Now (b) implies (a). 
We assume that (a) fails.  It follows immediately that there is an $f\in 2^{X}$ which is in the closure of the set of of $\phi(a,y)$ for $a\in M$, and there is $q\in S_{\phi^{*}}(M)$ such that for every neigbourhood $U$ of $q$ there is
$b\in M$, $tp_{\phi^{*}}(b/M)\in U$, such that $f(q) \neq f(tp_{\phi^{*}}(b/M))$.  Translating, and using Remark 2.1, this means that there are $a^{*}\in M^{**}$ and $b^{*}\in M^{*}$ such that
\newline 
(*)  for every $\phi^{*}$-formula $\psi(y)$ over $M$ satisfied 
by $b^{*}$ there is $b\in M$ satisfying $\psi(y)$  such that the value of $\phi(a^{*},b^{*})$ is different from that of $\phi(a^{*},b)$. Without loss of generality $\phi(a^{*},b^{*})$ is true. 

We now construct inductively $a_{n}, b_{n}\in M$ for $n=1,2,..$ such that
\newline
(i) $\models \phi(a_{i},b_{j})$ iff $i\leq j$,
\newline
(ii) $\models\neg\phi(a^{*}, b_{i})$ for all $i$,
\newline
(iii) $\models \phi(a_{i}, b^{*})$ for all $i$.
\newline
Suppose $a_{i}, b_{i}$ are constructed for $i\leq n$. As $tp_{\phi}(a^{*}/M^{*})$ is finitely satisfiable in $M$, choose $a_{n+1}\in M$ such that $\models \neg\phi(a_{n+1},b_{i})$ for $i\leq n$ and $\models \phi(a_{n+1},b^{*})$. 
Now using (*), let $b_{n+1}\in M$ be such that $\models\neg\phi(a^{*},b_{n+1})$ and $\models\phi(a_{i},b_{n+1})$ for $i\leq n+1$. So the construction can be carried out. (i) gives a contradiction to $\phi$ having not the order property in $M$. 
\end{proof}

The remaining material is more or less contained  in \cite{Ben-Yaacov}, although we spell some things out, especially Proposition 2.3 (c),  and offer some other proofs. 

\begin{Proposition} Conditions (a), (b) from Proposition 2.2. are also equivalent to each of
\newline
(c) Any $p(x)\in S_{\phi}(M)$ has an extension $p'(x)\in S_{\phi}(M^{*})$ which is both finitely satisfiable in and definable over $M$ where the $\phi$ definition of $p'$  is a $\phi^{*}$-formula over $M$,
\newline
(d) For any sequence $(a_{i})_{i<\omega}$ in $M$ there is a $\phi^{*}$-formula $\psi(y)$ over $M$ and a subsequence $(a_{j_{i}}:i < \omega)$ of the sequence $(a_{i})_{i}$ such that for every $b\in M^{*}$, the value of $\psi(b)$ equals the eventual value
of $\phi(a_{j_{i}},b)$.
\end{Proposition} 
\begin{proof} (a) implies (c): Given $p\in S_{\phi}(M)$ let $(a_{i})_{i}$ be a net $N$ in $M$ such that $tp_{\phi}(a_{i}/M)$ converges to $p$ in the space $S_{\phi}(M)$.  By (a) there is a subnet $N'$ of this net such that the functions $\phi(a_{i},y)$ converge to some $\phi^{*}$-formula $\psi(y)$ over $M$. Remember this means that for all $b\in M^{*}$, the value of $\psi(b)$\  is the ``eventual on $N'$" value of $\phi(a_{i},b)$.
So we obtain a complete 
$\phi$-type $p'$ over $M^{*}$ as follows: for $b\in M^{*}$, $\phi(x,b)\in p'$ if  ``eventually on $N'$" $\phi(a_{i},b)$ iff $\psi(b)$, and $\neg\phi(x,b)\in p'$ if eventually on $N'$, $\neg\phi(a_{i},b)$ iff $\neg\psi(b)$. We see that $p'$ is finitely satisfiable in $M$, definable over $M$ by $\psi$, and extends $p$. 
\newline
(c) implies (a). Let $f\in 2^{X}$ be in the closure of the set of functions $\phi(a,y)$ for $a\in M$. Let $a^{*}$ be as in Remark 2.1. So $tp_{\phi}(a^{*}/M^{*})$ is finitely satisfiable in $M$. Let $p$ be the restriction to $M$ of this type. Then we claim that $tp_{\phi}(a^{*}/M^{*})$ has to coincide with the global $\phi$-type $p'$ from (c). This is because by symmetry Proposition 2.2 also holds with $\phi^{*}$ in place of $\phi$, whereby $p$ has a unique coheir over $M^{*}$. But then it is easy to see that the $\phi$-definition $\psi(y)$ of $tp_{\phi}(a^{*}/M^{*})$ has to coincide with $f$.
\newline
(a) implies (d) is immediate because the sequence of functions $\phi(a_{i},y)$ has a subsequence which converges in $2^{X}$ to a $\phi^{*}$-formula $\psi(y)$ over $M$ and this will do the job.
\newline
(d) implies (b): This is as in the proof of (a) implies (b) in Proposition 2.2.  Namely from an example $a_{i}, b_{i}$ in $M$ witnessing the order property, extract a subsequence $\phi(a_{i},y)$ of functions convergent to a formula and get a contradiction.
\end{proof}

\begin{Remark} (Assume the equivalent conditions (a)-(d).) 
\newline
(i) Given $p(x)\in S_{\phi}(M)$, let $\psi(y)$ be the $\phi$-definition of $p$ (and also of its global coheir $p'$). Then there is a sequence $(a_{i}:i<\omega)$ in $M$ such that for any $b\in M$, $\psi(b)$ holds iff eventually $\phi(a_{i},b)$ holds, and $\neg\psi(b)$ holds iff eventually $\neg\phi(a_{i},b)$ holds.
\newline
(ii)  The formula $\psi(y)$ from (i) is equivalent to a finite positive Boolean combination of formulas $\phi(a,y)$ for $a\in M$.
\end{Remark}
\begin{proof}
(i)  Let $M_{0}$ be a countable elementary substructure of the reduct of $M$ to $\phi(x,y)$ which contains the defining parameters of $\psi(y)$. 
Let $(a_{i}:i<\omega)$ be a sequence in $M_{0}$ such that $tp_{\phi}(a_{i}/M_{0})$ converges to $p_{0} = p|M_{0}$. As $\phi(x,y)$ does not have the order property in $M_{0}$, we may assume, from condition (d) in Proposition 2.3, that the formulas $\phi(a_{i},y)$ converge to the defining formula $\psi(y)$ of $p_{0}$ (so also of $p$ and $p'$) in the space $2^{S_{\phi^{*}}(M_{0})}$.  This suffices. 
\newline
(ii)  This follows from (i) by compactness. Specifically, in the saturated model $M^{*}$, $\models\psi(b)$ holds iff for some $n$, $\models\phi(a_{i},b)$ for all $i\geq n$, namely $\psi(y)$ is equivalent to a certain infinite disjunction of infinite conjunctions (of the $\phi(a_{i},y)$).  An easy compactness argument yields the equivalence of $\psi(y)$ with a finite subdisjunction of finite subconjunctions. 
\end{proof}

\end{document}